\documentclass{article}
\usepackage[utf8]{inputenc}
\usepackage{maa-monthly}

\theoremstyle{theorem}
\newtheorem{theorem}{Theorem}
 
\theoremstyle{definition}

\begin{document}

\title{A Prime-Representing Constant}
\markright{A Prime-Representing Constant}
\author{Dylan Fridman, Juli Garbulsky, Bruno Glecer,\\ James Grime, and Massi Tron Florentin}

\maketitle

\begin{abstract}
We present a constant and a recursive relation to define a sequence $f_n$ such that the floor of $f_n$ is the $n$th prime. Therefore, this constant generates the complete sequence of primes. We also show this constant is irrational and consider other sequences that can be generated using the same method.
\end{abstract}

In this note we present a constant and a recursive relation that generates the complete sequence of primes. 

\begin{theorem}\label{thm1} Let $p_n$ denote the $n$th prime. Then there exists a constant \[f_1 = 2.920050977316\ldots\] and a sequence
\[ f_{n} = \lfloor f_{n-1} \rfloor ( f_{n-1} - \lfloor f_{n-1} \rfloor + 1 ) \] 
such that the floor of $f_n$ is the $n$th prime, i.e., $\lfloor f_n \rfloor = p_n.$ \end{theorem}

\begin{proof} Define a sequence $g_n$ as follows:
\[ g_n = \sum_{k=1}^n \frac{p_k - 1}{\prod_{i=1}^{k-1} p_i}. \]

Bertrand's postulate gives bounds on the size of each prime, namely $p_{n-1} < p_n < 2p_{n-1} - 1.$ We use this as we consider $g_n$ as $n$ tends to infinity:
\begin{eqnarray*}
\lim_{n \to \infty} g_n &=& (p_1 - 1) + \frac{p_2 - 1}{p_1} + \frac{p_3 - 1}{p_2p_1} + \frac{p_4 - 1}{p_3p_2p_1} +  \cdots \\
&<& (p_1 - 1) + \frac{2p_1 - 2}{p_1} + \frac{2p_2 - 2}{p_2p_1} + \frac{2p_3 - 2}{p_3p_2p_1} + \cdots.
\end{eqnarray*}

Most terms on the right-hand side of the inequality cancel out, leaving $(p_1 - 1) + 2 = 3.$ Since $g_n$ is strictly increasing and bounded, it is convergent. Its value is $2.920050977316\ldots.$ 

Define $f_1 = \lim_{n \to \infty} g_n$ and 
\begin{eqnarray*} f_{n} &=& (f_1 - g_{n-1}) \prod_{i=1}^{n-1} p_i\\  &=& (p_{n} - 1) + \frac{p_{n+1} -1}{p_{n}} + \frac{p_{n+2} -1}{p_{n+1}p_{n}} + \frac{p_{n+3} -1}{p_{n+2}p_{n+1}p_{n}} + \cdots. \end{eqnarray*}
It can be determined that $f_{n} = p_{n-1}(f_{n-1} - p_{n-1} + 1).$ 

Using Bertrand's postulate again, we have
\[ f_{n} > (p_{n} - 1) + \frac{p_{n} -1}{p_{n}} + \frac{p_{n+1} -1}{p_{n+1}p_{n}} + \frac{p_{n+2}-1}{p_{n+2}p_{n+1}p_{n}} + \cdots \]
and
\[ f_{n} < (p_{n} - 1) + \frac{2p_{n} -2}{p_{n}} + \frac{2p_{n+1} -2}{p_{n+1}p_{n}} + \frac{2p_{n+2} -2}{p_{n+2}p_{n+1}p_{n}} + \cdots. \]

\noindent
Most terms in the inequalities cancel out, leaving $p_n <  f_{n} < p_{n} + 1,$ and so the floor function $\lfloor f_{n} \rfloor = p_{n}.$ The sequence $f_n$ can now be generated recursively with the formula $f_{n} = \lfloor f_{n-1} \rfloor (f_{n-1} - \lfloor f_{n-1} \rfloor + 1).$
\end{proof}

There are many other examples of prime-producing formulae. One of the earliest and most famous examples of a prime-producing formula is the polynomial discovered by Euler \[n^2 + n + 41.\] This simple polynomial gives prime values for each integer $n$ from 0 to 39.

There are also many examples of prime-generating constants. For example, if $ \alpha = \sum_{i \geqslant 1} p_i/10^{2^{i+1}},$ then $\alpha$ is a prime-generating constant because we may generate the $n$th prime with the formula \[ p_n \lfloor 10^{2^{n+1}} \alpha \rfloor - 10^{2^{n}} \lfloor 10^{2^{n}} \alpha \rfloor. \] It is clear that $\alpha$ is irrational since it is given as a nonrepeating decimal. See \cite{CP} Theorem 1.2.2.

We took inspiration for Theorem \ref{thm1} from Mills's constant as described by William H. Mills in 1947 in \cite{M}. In this one-page note, Mills defines a constant $A$ and a function $\lfloor A^{3^n} \rfloor$ in such a way that the value of the function is prime for all natural numbers $n.$ 

Mills's constant is defined as the smallest positive real number that may be used for $A.$ Mills did not give any specific value for $A$; however, if the Riemann hypothesis is true, then we may define $A = 1.306377883863\ldots,$ which generates the following sequence: \[2, 11, 1361, 2521008887, 16022236204009818131831320183, \ldots.\] Each number in this sequence is prime. However, this is far from a complete sequence of primes. Indeed, the next value in this sequence is of the order $10^{84}.$ It is not yet known whether Mills's constant is irrational.

\begin{theorem} \label{thm2} The prime-generating constant $f_1$ is irrational. \end{theorem}
\begin{proof} 
Since $p_n < f_n < p_n + 1$ for all $n,$ we may write $f_n = p_n + r_n,$ where $0 < r_n < 1.$ 

Assume $f_1$ is rational, so that $f_1 = a/b.$ Using the recurrance relation 
\[ f_{n+1} = p_n(f_n - p_n + 1) = p_n(r_n+1), \] 
we see that $bf_n$ is an integer for all $n.$ In particular, $r_n \geqslant 1/b$ for all $n.$

However, we can rearrange the above expression as follows:
\[ (r_n + 1) = \frac{f_{n+1}}{p_n} = \frac{p_{n+1} + r_{n+1}}{p_n}. \]
By the prime number theorem we know that $p_{n+1}/p_n$ tends to 1 as $n$ tends to infinity. This means that, since the $r_n$ are bounded, the right-hand side also tends to 1. And so $\lim_{n \to \infty} r_n = 0.$ This contradicts $r_n \geqslant 1/b$ for all $n.$
\end{proof}

The prime-generating constant of Theorem \ref{thm1} and Theorem \ref{thm2} has appeared before, albeit in a different context, in \cite{BM}, \cite{S}, and \cite{R}. In these papers the constant is described as the average of the sequence $2, 3, 2, 3, 2, 5, 2, 3, \ldots,$ the sequence of smallest primes that do not divide $n.$

To prove that the average of this sequence is the same as our constant, let us consider the probability that a prime $p_k$ is the smallest prime that does not divide $n,$ for some natural number $n.$ This probability can be written as 
\[ P(p_k \textrm{ does not divide } n \textrm{ and } p_1, p_2, \ldots, p_{k-1} \textrm{ divide } n) = \left( 1 - \frac{1}{p_k} \right) \prod_{i=1}^{k-1} \frac{1}{p_i}.  \]
The average of the sequence is therefore given by \[ \sum_{k=1}^{\infty} P(\textrm{The smallest prime that does not divide } n \textrm{ is } p_k) \cdot p_k = \sum_{k=1}^\infty \frac{p_k -1}{\prod_{i=1}^{k-1} p_i},\] which is the definition of our prime-generating constant in Theorem \ref{thm1}.

The arguments in this note are not limited to the sequence of primes. Indeed, they can be applied to any sequence that follows Bertrand's postulate, and any such sequence will define its own constant $f_1.$ This constant can then be used in the same formula for $f_n$ described in Theorem \ref{thm1}. 

Similarly, if a sequence follows Bertrand's Postulate and the ratio of consecutive terms tend to 1, then the sequence-generating constant is irrational by the same argument used in the proof of Theorem \ref{thm2}.

If we just consider the lower bound of Bertrand's postulate, the most compact sequence we get is $2, 3, 4, 5, 6, 7, \ldots.$ In this case the sequence defines the Euler constant, that is to say $f_1 = e.$ Since the ratio of consecutive integers tends to 1 this gives us another proof that $e$ is irrational.

At the other extreme, if we just  consider the upper bound of Bertrand's postulate, the most compact sequence we get is $3, 4, 6, 10, 18, 34, \ldots.$ Here the $n$th term is $2^{n-1} + 2$ and the sequence defines the constant $3.56797609098\ldots.$ Interestingly, we have not found this constant mentioned anywhere in the literature before. 

\begin{acknowledgment}{Acknowledgments.}
Dylan, Juli, Bruno and Massi are a group of 18/19-year-old friends from Buenos Aires, Argentina. The original idea came to Juli while having a shower. Bruno calculated the prime-generating constant, first by brute-force and then by finding its formula. As the investigation continued, Juli and Bruno were joined by Massi and Dylan. Later, the team contacted mathematician James Grime who helped by tidying up some of the proofs and writing this note.  So $\infty$ thanks to James!
\end{acknowledgment}

\begin{biog}
\item[Dylan Flidman]
\begin{affil}
Mathematics Department (FCEN), University of Buenos Aires, Acceso Pabell$\acute{o}$n 1, Buenos Aires, Argentina \\
dylanfridman@gmail.com
\end{affil}
\item[Juli Garbulsky]
\begin{affil}
Mathematics Department (FCEN), University of Buenos Aires, Acceso Pabell$\acute{o}$n 1, Buenos Aires, Argentina  \\
juliangarbulsky@gmail.com
\end{affil}
\item[Bruno Glecer]
\begin{affil}
National Technological University - Buenos Aires Regional Faculty - Department of Electronic Engineering, Argentina\\
bruno.glecer@gmail.com
\end{affil}
\item[James Grime]
\begin{affil}
Cambridge, UK \\
email@jamesgrime.com
\end{affil}
\item[Massi Tron Florentin]
\begin{affil}
Physics Department (FCEN), University of Buenos Aires, Acceso Pabell$\acute{o}$n 1, Buenos Aires, Argentina  \\
tronmassimiliano@gmail.com
\end{affil}
\end{biog}
\vfill\eject

\end{document}